\documentclass{elsart1p}
\usepackage{amsmath,amssymb}
\usepackage{graphicx}
\usepackage{subfigure,color}
\usepackage{epsfig}
\usepackage[misc]{ifsym}

\usepackage{color}
\newcommand{\abs}[1]{\left\vert#1\right\vert}
\newcommand{\norm}[1]{\left\Vert#1\right\Vert}
\newtheorem{theorem}{Theorem}[section]

\newenvironment{proof}[1][Proof]{\begin{trivlist}
\item[\hskip \labelsep {\bfseries #1}]}{\end{trivlist}}

\begin{document}

\begin{frontmatter}


\title{An energy preserving finite difference scheme for the Poisson-Nernst-Planck system}


\author[Tongji]{Dongdong He\ead{dongdonghe@tongji.edu.cn}}\and
\author[Zhongnan,UNCC]{Kejia Pan\corauthref{cor}\ead{pankejia@hotmail.com}}
\corauth[cor]{Corresponding author.}
\address[Tongji]{School of Aerospace Engineering and Applied Mechanics,
Tongji University, Shanghai 200092, China.}
\address[Zhongnan]{ School of Mathematics and Statistics, Central South University, Changsha 410083, China}
\address[UNCC]{Department of Mathematics and Statistics, University of North Carolina at Charlotte, Charlotte, NC 28223, USA}
\begin{abstract}
In this paper,
we construct a semi-implicit finite difference method for the time dependent Poisson-Nernst-Planck system. Although the Poisson-Nernst-Planck system is a nonlinear system, the numerical method presented in this paper only needs to solve a linear system at each time step, which can be done very efficiently. The rigorous proof for the mass conservation and  electric potential energy decay are shown.  Moreover, mesh refinement analysis shows that the method is second order convergent in space and first order convergent in time.
Finally we point out that our method can be easily extended to the case of multi-ions.
\end{abstract}

\begin{keyword}
Poisson-Nernst-Planck system, finite difference method, mass
conservation, ion concentration positivity, energy decay
\end{keyword}

\end{frontmatter}
\newpage
\section{Introduction}

The classical unsteady dimensionless drift-diffusion system which
describes the evolution of positive and negative charged particles
$p, n$, and electric potential $\phi$ is given as
follows~\cite{Prohl}
\begin{equation}\label{potential}
\left\{ \begin{aligned}
  p_{t}&=\nabla\cdot(\nabla p+p\nabla\phi), \ \textrm{in}\  \Omega_T:=(0,T]\times\Omega,\\
n_{t}&=\nabla\cdot(\nabla n-n\nabla\phi), \ \textrm{in}\ \Omega_T,\\
-\Delta\phi&=p-n, \ \textrm{in}\ \Omega_T,
        \end{aligned} \right.
\end{equation}
where $\Omega$ is a bounded domain, $[0,T]$ is the time interval, the length scale is chosen as the
Debye length,  and the time scale is chosen as the diffusive time
scale. Note that
the Debye length is much smaller than the physical characteristic
length in most cases.  Thus, if the length scale is chosen as the
physical characteristic length scale in these cases, there will be a small
parameter in front of the electric potential
term in the Poisson equation (\ref{potential}), which will result in  a singular perturbation
problem~\cite{Chen, Wang, Singer,GILLESPIE}. However, in this paper, we only consider
the case that the characteristic length scale is the same order of
the Debye length, which also has some applications. For example, inside the ion channel in the cell
membrane, the characteristic length scale of the ion channel is the
same order of the Debye length. In this situation, we could choose the Debye length as
the length scale so that the dimensionless system (\ref{potential}) is meaningful. The above system is called Poisson-Nernst-Planck system,
which was first formulated by W. Nernst and M. Planck to describe
the potential difference in a galvanic cell. The system has lots of
applications in electrochemistry~\cite{Roubicek},
biology~\cite{Eisenberg} and semiconductors~\cite{Jerome,Biler,Gajewski}. Based on the analytical derivations for the energy and entropy laws, Schmuck theoretically proved the existence, and in some cases uniqueness of the weak solution for the more general
context of the Navier-Stokes-Nernst-Planck-Poisson system \cite{Schmuck}. In this paper, we focus on the numerical solution of the Poisson-Nernst-Planck system (\ref{potential}).

When numerically solving the PDEs, to keep the original
physical feature is greatly important in constructing numerical
schemes for different physical problems. For example, one successful
and active research is to construct structure-preserving scheme for
the ODE systems (see~\cite{Hairer} and references therein). Only
schemes that are carefully designed can preserve mass and energy conservative properties. 
For example,
finite difference methods were developed for solving the Euler Equations and Burgers equations to preserve the discrete energy dynamics \cite{Fisher}. In \cite{Hof}, authors used a finite volume method for the shallow water equations which conserves the mass, momentum and energy of the system.
A finite difference method was presented in \cite{Li} for solving the nonlinear Klein-Gordon equation which preserves the total energy.
Qiao et al. \cite{Qiao} showed an unconditionally stable finite difference scheme for the dynamics of the molecular beam epitaxy, where the scheme preserves the energy decay rate exactly at discrete level.
In \cite{Celledoni}, authors developed a general method of discretizing PDEs that preserving the energy using the average vector field method. Chiu et al.  \cite{Chiu} developed a general mesh-free scheme for solving PDEs that can preserve the energy at discrete level. Chang et al.~\cite{Chang}
 discussed conservative and nonconservative properties of eight finite difference schemes for solving the generalized
nonlinear Schr$\ddot{\textrm{o}}$dinger equation. Chen et al.~\cite{Wenbin1, Wenbin2} proposed energy conservative finite difference schemes for solving the 2D and 3D Maxwell equations, respectively.

In the past several decades, there appears a wide range of literature on
numerical methods for the Poisson-Nernst-Planck system, including
finite difference method, finite element method, and finite volume
method, see~\cite{Snowden,Brezzi,Douglas,Miller,Peng,Liu,Liu2004,Bessemoulin,Flavell,LiuHL} and references
therein. Here we just mention some of the recent work, Bessemoulin-Chatard~\cite{Bessemoulin} gave a conservative  finite volume method for solving the drift-diffusion equation, where the entropy inequality is preserved. Flavell et al.~\cite{Flavell} and Liu et al.~\cite{LiuHL} constructed  two different conservative finite difference methods which satisfy the mass preserving, ion concentration positivity as well as total free energy dissipation numerically, where the total free energy is related to both electric potential and ion concentration, which is called entropy in \cite{Prohl}.
And Prohl et al.~\cite{Prohl} presented two different finite element methods which satisfy  electric potential energy decay and entropy decay properties, respectively.
 Now we briefly illustrate the first part of the work in~\cite{Prohl} as follows:
under the following initial conditions and zero Neumann boundary conditions
\begin{equation}\label{initial}
p(0,\vec{x})=p_{0}(\vec{x})\geq0, \  n(0,\vec{x})=n_{0}(\vec{x})\geq0,
\ \textrm{in}\ \Omega,
\end{equation}
\begin{equation}\label{boundary}
\frac{\partial \phi}{\partial \vec{n}}=\frac{\partial n}{\partial \vec{n}}=\frac{\partial p}{\partial \vec{n}}=0,  \ \textrm{on}\
\partial\Omega_T:=(0,T]\times \partial \Omega.
\end{equation}
It is well known~\cite{Biler} that the non-negative $p$, $n$ is conserved in $\Omega_T$, and the system (\ref{boundary}) satisfies
mass conservation, that is, for any $t\in (0,T]$,
\begin{equation}\label{mass}
M_{p}\equiv\int_{\Omega}p_{0}(\vec{x})d\vec{x}=\int_{\Omega}p(t,\vec{x})d\vec{x},\;
M_{n}\equiv\int_{\Omega}n_{0}(\vec{x})d\vec{x}=\int_{\Omega}n(t,\vec{x})d\vec{x},
\end{equation}
where $M_{p}$, $M_{n}$ are two positive constants which  must be the same,
since from (\ref{potential}) and (\ref{boundary}) we have
$$M_{p}-M_{n}=\int_{\Omega}(p-n)dx=-\int_{\partial\Omega}\frac{\partial \phi}{\partial \vec{n}}ds=0.$$
And it is also shown in~\cite{Biler} and \cite{Schmuck} that the system satisfies the following energy law
\begin{equation}\label{energy1}
E(t)+\int^{t}_{0}\int_{\Omega}(({p-n})^{2}+(p+n)\abs{\nabla\phi}^{2}) d\vec{x}dt=E(0),
\end{equation}
where $E(t)=\frac{1}{2}\int_{\Omega}|\nabla\phi|^{2} d\vec{x}$ is
the electric potential energy.
The above energy law~(\ref{energy1}) can be rewritten as
\begin{equation}\label{energy}
\frac{dE}{dt}=-\int_{\Omega}((p-n)^{2}+(p+n)|\nabla\phi|^{2})
d\vec{x}.
\end{equation}

 Prohl et al.~\cite{ Prohl} proposed a  finite element method which can preserve the mass conservation (\ref{mass}), ion concentration positivity and electric potential energy decay (\ref{energy1}) in \cite{Prohl}.
However, the scheme in \cite{Prohl} is fully implicit, one has to solve a nonlinear system at each time step.
In~\cite{Prohl},  a fixed point iteration method is used to solve the nonlinear system at each time step in order to get the rigorous physical quantities preserving results. In this paper, we present a simple
semi-implicit finite difference method for the Poisson-Nernst-Planck
system.  For the new scheme, the unknown variables at next time step form a linear
system which can be solved efficiently, no iteration is needed. Furthermore,
the new scheme preserves mass conservation and electric potential energy identity
numerically. Numerical results confirm the above properties.
Mesh refinement analysis shows that the method is second order convergent in space and  first order convergent in time. Finallly, we point out that our method can be extended to the case of multi-ions without any difficulty.

 The rest of the paper is organized as follows, section \ref{sec:3} gives the detailed numerical scheme and its properties,
 section \ref{sec:4} discusses the extension of the method for the case of multi-ions,
 section \ref{sec:5} shows the numerical results, and conclusions and discussions are given in the final section.

\section{Numerical method}\label{sec:3}

In this section, we will develop a finite difference method which
can guarantee the mass conservation (\ref{mass}) and energy decay
(\ref{energy}) numerically.

Although the method presented in the following can be extended into
three dimension without any difficulty, we only give a detailed
description when $\Omega$ is a two dimensional rectangular domain,
i.e. $\Omega=[a,b]\times[c,d]$. Let $N_{x},N_{y}$ be positive
integers, the domain $\Omega$ is uniformly partitioned with
$\Delta x=\frac{b-a}{N_{x}}, \Delta y=\frac{d-c}{N_{y}}$ and variables are stored at each
cell center as follows
\begin{align}
\Omega_{h}=\{(x_{j},y_{k})|x_{j}=a+(j-\frac{1}{2})\Delta
x,y_{k}=c+(k-\frac{1}{2})\Delta y, 1\leq j\leq N_{x}, 1\leq k\leq
N_{y}\}.\nonumber
\end{align}
And time step is denoted by $\Delta t$.

 For a given two-dimensional grid function $f_{j,k}$, we define
the following difference operators:
\begin{align*}
&\nabla_{h}f_{j,k}=(\frac{f_{j+1,k}-f_{j,k}}{\Delta
x},\frac{f_{j,k+1}-f_{j,k}}{\Delta y}),\\
&\Delta_{h}f_{j,k}=\frac{f_{j+1,k}-2f_{j,k}+f_{j-1,k}}{(\Delta
x)^{2}}+\frac{f_{j,k+1}-2f_{j,k}+f_{j,k-1}}{(\Delta
y)^{2}}, \\
&\delta_{x}f_{j,k}=\frac{f_{j+1/2,k}-f_{j-1/2,k}}{\Delta x}, \;\;\delta_{y}f_{j,k}=\frac{f_{j,k+1/2}-f_{j,k-1/2}}{\Delta y},\\
&\partial^{h}_{x}f_{j,k}=\frac{f_{j,k}-f_{j-1,k}}{\Delta
x},\;\; \partial^{h}_{y}f_{j,k}=\frac{f_{j,k}-f_{j,k-1}}{\Delta y},
\end{align*}
where
\begin{align}
f_{j+1/2,k}=\frac{f_{j,k}+f_{j+1,k}}{2},\;\; f_{j,k+1/2}=\frac{f_{j,k}+f_{j,k+1}}{2}.\nonumber
\end{align}

The discrete $L^{2}$ inner product and the discrete $L^{2}$ norm are
defined as
\begin{align}
<f,g>_{h}=\sum^{N_{x}}_{j=1}\sum^{N_{y}}_{k=1}f_{j,k}g_{j,k}\Delta
x\Delta y,\quad \parallel f\parallel^{2}_{h}=<f,f>_{h}.
\nonumber
\end{align}
We also define
\begin{equation*}
<\nabla f,\nabla g>_{h}=\sum^{N_{x}}_{j=1}\sum^{N_{y}}_{k=1}\nabla_{h}f_{j,k}\cdot\nabla_{h}g_{j,k}\Delta
x\Delta y,\quad \parallel \nabla f\parallel^{2}_{h}=<\nabla f, \nabla f>_{h}.
\end{equation*}

\subsection{Description of the method}

For zero Neumann boundary conditions (\ref{boundary}), we define the
values on center of the fictitious cells outside the boundary as
follows
\begin{align}\label{0_Neumann}
p_{0,k}=p_{1,k}, p_{j,0}=p_{j,1}, p_{N_{x}+1,k}=p_{N_{x},k},
p_{j,N_{y}+1}=p_{j,N_{y}},\nonumber\\
n_{0,k}=n_{1,k}, n_{j,0}=n_{j,1}, n_{N_{x}+1,k}=n_{N_{x},k},
n_{j,N_{y}+1}=n_{j,N_{y}},\\
\phi_{0,k}=\phi_{1,k},
\phi_{j,0}=\phi_{j,1}, \phi_{N_{x}+1,k}=\phi_{N_{x},k},
\phi_{j,N_{y}+1}=\phi_{j,N_{y}}.\nonumber
\end{align}
where values on the the fictitious cells are denoted by the
subscript with $0$, $N_{x}+1$, and $N_{y}+1$.


Our scheme for the Poisson-Nernst-Planck system (\ref{potential}) is as follows
\begin{align} \label{p_discrete}
\frac{p^{m+1}_{j,k}-p^{m}_{j,k}}{\Delta
t}=\Delta_{h}{p^{m+1/2}_{j,k}}+\left(\delta_x(p^m_{j,k}\delta_x\phi^{m+1/2}_{j,k})+\delta_y(p^m_{j,k}\delta_y\phi^{m+1/2}_{j,k})\right),
\nonumber \\\ \textrm{for} \ j=1\cdots N_{x}, k=1\cdots N_{y}, m\geq0,
\end{align}
\begin{align} \label{n_discrete}
\frac{n^{m+1}_{j,k}-n^{m}_{j,k}}{\Delta
t}=\Delta_{h}{n^{m+1/2}_{j,k}}-\left(\delta_x(n^m_{j,k}\delta_x\phi^{m+1/2}_{j,k})+\delta_y(n^m_{j,k}\delta_y\phi^{m+1/2}_{j,k})\right),
\nonumber \\\ \textrm{for}\ j=1\cdots N_{x}, k=1\cdots N_{y}, m\geq0,
\end{align}
\begin{equation} \label{phi_discrete}
-\Delta_{h}\phi^{m+1}_{j,k}=p^{m+1}_{j,k}-n^{m+1}_{j,k},\ \ \textrm{for}
\  j=1\cdots N_{x}, k=1\cdots N_{y}, m\geq-1,
\end{equation}
where $$\phi^{m+1/2}_{j,k}=\frac{\phi^{m}_{j,k}+\phi^{m+1}_{j,k}}{2}, p^{m+1/2}_{j,k}=\frac{p^{m}_{j,k}+p^{m+1}_{j,k}}{2}, n^{m+1/2}_{j,k}=\frac{n^{m}_{j,k}+n^{m+1}_{j,k}}{2}.$$

We should mention that the discrete equation for electric potential
(\ref{phi_discrete}) is used for $m=-1$, since there is no initial
conditions for $\phi$, see (\ref{initial}).

\subsection{Implementation of the finite difference method}
The method (\ref{p_discrete})-(\ref{phi_discrete}) is a
semi-implicit method, which can be implemented efficiently. Assume
that the quantities $p,n,\phi$ are known at the previous time step
$m$, and rewrite equations (\ref{p_discrete})-(\ref{phi_discrete}) in matrix
and vector form as follows
\begin{align} \label{p_matrix}
(\frac{I}{\Delta t}-\frac{F}{2})P^{m+1}=(\frac{I}{\Delta
t}+\frac{F}{2})P^{m}+A(P^{m})\frac{\Phi^{m+1}+\Phi^{m}}{2},
\end{align}
\begin{align} \label{n_matrix}
(\frac{I}{\Delta t}-\frac{F}{2})N^{m+1}=(\frac{I}{\Delta
t}+\frac{F}{2})N^{m}-A(N^{m})\frac{\Phi^{m+1}+\Phi^{m}}{2},
\end{align}
\begin{equation} \label{phi_matrix}
F\Phi^{m+1}=N^{m+1}-P^{m+1},
\end{equation}
where $I$ is the identity matrix, $F$ is the matrix form of discrete
Laplacian operator $\Delta_{h}$, $A(P^{m})$ is the coefficient
matrix obtained from the second term of (\ref{p_discrete}), which is a linear operator for
$P^{m}$, and $P^{m+1},N^{m+1},\Phi^{m+1}$ are vector form of
$p,n,\phi$ at time step $m+1$.

As we can see that (\ref{p_matrix})-(\ref{phi_matrix}) is actually a
linear system of unknowns $P^{m+1}$, $N^{m+1}$ and $\Phi^{m+1}$, we
can eliminate $P^{m+1}$ and $N^{m+1}$ in (\ref{phi_matrix}) using
(\ref{p_matrix}) and (\ref{n_matrix}), this yields
\begin{equation} \label{phi_matrix_matrix}
((\frac{2}{\Delta
t}I-F)F+A(P^{m}+N^{m}))\Phi^{m+1}=(\frac{2}{\Delta
t}I+F)(N^{m}-P^{m})-A(P^{m}+N^{m})\Phi^{m}.
\end{equation}

Once $\Phi^{m+1}$ is obtained, we can
obtain $P^{m+1}$ and $N^{m+1}$ by solving (\ref{p_matrix}) and
(\ref{n_matrix}), respectively.

As mentioned in the introduction, $p$, $n$ are non-negative in the entire time interval $[0, T]$ as theoretically shown in~\cite{Biler}. In addition, through lots of numerical tests, it is found that the numerical solutions of $P^m, N^m$ from the scheme (\ref{p_discrete})-(\ref{phi_discrete}) are also non-negative. If we make an assumption that the finite difference approximation of the derivative of $\phi^m$ is uniformly bounded in the entire computational time as done in~\cite{Flavell} (see equation (37) in~\cite{Flavell}), then we can prove that the numerical solutions of $P^m, N^m$ are also non-negative under some restriction of the time step size and mesh size followed by a similar proof of~\cite{Flavell}. However, in general,  it is not suitable to make such a  prior assumption for the numerical solution which involves the quantities in the next time step to prove the properties of the numerical solutions. Thus, it will be a very hard task to prove the non-negativity of $P^m, N^m$ without the prior assumption for the finite difference approximation of the derivative of $\phi^m$. But a lot of numerical tests show that the numerical scheme (\ref{p_discrete})-(\ref{phi_discrete}) produces non-negative $P^m, N^m$. Thus, in the following discussion, we simply assume that the numerical solutions $P^m, N^m$ are non-negative.

\begin{theorem}\label{thm0}
Within the numerical solution of $\Phi^{m+1}$ up to a constant, the numerical solutions $P^{m+1}, N^{m+1}, \Phi^{m+1}$ of (\ref{p_discrete})-(\ref{phi_discrete}) are unique.
\end{theorem}
\begin{proof}
From (\ref{p_discrete})-(\ref{phi_discrete}), we can see that $F$, $A(P^{m})$ and $A(N^{m})$ are all symmetric banded and have the same matrix element structure. Moreover, for each of these three matrices, the diagonal elements are negative while the sum of each row is zero, we can get all eigenvalues of each matrix are less than or equal to zero through Gerschgorin Circle Theorem \cite{Golub}. Thus all these three matrices are all negative semi-definite. Indeed, each of these three matrices has exactly one zero eigenvalue and all other negative eigenvalues. Furthermore, we can get $\frac{2I}{\Delta t}-F$ is positive definite. Since $F$ is negative semi-definite, $\frac{2I}{\Delta t}-F$ is positive definite,  and $F$  can be exchanged with $\frac{2I}{\Delta t}-F$, we have $(\frac{2I}{\Delta t}-F)F$ is negative semi-definite. And from above, we have $A(P^{m}+N^{m})$ is negative semi-definite. Thus $(\frac{2I}{\Delta t}-F)F+A(P^{m}+N^{m})$, the coefficient matrix of (\ref{phi_matrix_matrix}), is negative semi-definite. And it is easy to see that $(\frac{I}{\Delta t}-\frac{F}{2})F+A(P^{m}+N^{m})$ has exactly one zero eigenvalue and all other negative eigenvalues. Therefore within the numerical solution of $\Phi^{m+1}$ up to a constant, the numerical solution of (\ref{phi_discrete}) is unique.

Once we get $\Phi^{m+1}$, we can get $P^{m+1}$ and $N^{m+1}$ by solving (\ref{p_matrix}) and
(\ref{n_matrix}), respectively.  Since $\frac{I}{\Delta t}-\frac{F}{2}$ is positive definite, the numerical solutions of $P^{m+1}$ and $N^{m+1}$ are unique.
This completes the proof. \qed
\end{proof}

Since the coefficient matrix of (\ref{phi_matrix_matrix}) is negative semi-definite, symmetric and banded while the coefficient matrix of (\ref{p_matrix}) and (\ref{n_matrix}) is positive definite, symmetric and banded, these linear systems can be numerically solved very efficiently.

In numerical computation using (\ref{p_discrete})-(\ref{phi_discrete}), we set $\phi^{m+1}$ to be zero at one boundary point at each time step for (\ref{phi_discrete}) so that the $\phi^{m+1}$ is uniquely determined.

\subsection{Main properties of the numerical scheme}

\begin{theorem}\label{thm1}
For the solutions of (\ref{p_discrete})-(\ref{phi_discrete}), the
discrete form of mass conservation (\ref{mass}) holds, that is, for
any $m\geq0$,
\begin{equation} \label{mass_discrete}
\sum^{N_{x}}_{j=1}\sum^{N_{y}}_{k=1}p^{m+1}_{j,k}\Delta
x\Delta
y=\sum^{N_{x}}_{j=1}\sum^{N_{y}}_{k=1}p^{m}_{j,k}\Delta
x\Delta y,
\end{equation}
and
\begin{equation}
\sum^{N_{x}}_{j=1}\sum^{N_{y}}_{k=1}n^{m+1}_{j,k}\Delta
x\Delta
y=\sum^{N_{x}}_{j=1}\sum^{N_{y}}_{k=1}n^{m}_{j,k}\Delta
x\Delta y.
\end{equation}
\end{theorem}

\begin{proof}
  Multiplying $\Delta t\Delta x\Delta y$ to both sides of
(\ref{p_discrete}) and summing for $j=1\cdots N_{x}, k=1\cdots
N_{y}$, and applying the boundary conditions (\ref{0_Neumann}), we get
\begin{align}
&\sum^{N_{x}}_{j=1}\sum^{N_{y}}_{k=1}p^{m+1}_{j,k}\Delta
x\Delta
y-\sum^{N_{x}}_{j=1}\sum^{N_{y}}_{k=1}p^{m}_{j,k}\Delta
x\Delta y\nonumber\\
& =\Delta
t\Delta
x \Delta
y\sum^{N_{x}}_{j=1}\sum^{N_{y}}_{k=1}\Delta_{h}{p^{m+1/2}_{j,k}}+\Delta
t\Delta x \Delta y\sum^{N_{x}}_{j=1}\sum^{N_{y}}_{k=1}\delta_x(p^m_{j,k}\delta_x\phi^{m+1/2}_{j,k})\nonumber\\
&+\Delta t\Delta x \Delta y\sum^{N_{x}}_{j=1}\sum^{N_{y}}_{k=1}\delta_y(p^m_{j,k}\delta_y\phi^{m+1/2}_{j,k}).\label{Massid}
\end{align}
From (\ref{0_Neumann}), the first term of the right hand side of (\ref{Massid}) is zero obviously. The second term is
\begin{align}
&\Delta t\Delta x \Delta y\sum^{N_{x}}_{j=1}\sum^{N_{y}}_{k=1}\delta_x(p^m_{j,k}\delta_x\phi^{m+1/2}_{j,k})\nonumber\\
&=\frac{\Delta t \Delta y}{\Delta x}\sum^{N_{y}}_{k=1}\sum^{N_{x}}_{j=1}\left(p^{m}_{j+\frac{1}{2},k}(\phi^{m+\frac{1}{2}}_{j+1,k}-\phi^{m+\frac{1}{2}}_{j,k})-
p^{m}_{j-\frac{1}{2},k}(\phi^{m+\frac{1}{2}}_{j,k}-\phi^{m+\frac{1}{2}}_{j-1,k})\right)\nonumber\\
&=\frac{\Delta t \Delta y}{\Delta x}\sum^{N_{y}}_{k=1}\left(p^{m}_{N_{x}+\frac{1}{2},k}(\phi^{m+\frac{1}{2}}_{N_{x}+1,k}-\phi^{m+\frac{1}{2}}_{N_{x},k})-
p^{m}_{\frac{1}{2},k}(\phi^{m+\frac{1}{2}}_{1,k}-\phi^{m+\frac{1}{2}}_{0,k})\right).
\end{align}
For zero Neuman boundary conditions, we have $\phi^{m+\frac{1}{2}}_{1,k}=\phi^{m+\frac{1}{2}}_{0,k}$ and $\phi^{m+\frac{1}{2}}_{N_x+1,k}=\phi^{m+\frac{1}{2}}_{N_x,k}$. Thus the second term is zero. Similarly,  the third term is also zero. Therefore, the mass conservation identity for $p$ is proved.

Similar proof can be used for $n$. This completes the proof. \qed
\end{proof}

\begin{theorem}\label{thm3}
For the solutions of (\ref{p_discrete})-(\ref{phi_discrete}), the
discrete form of energy identity (\ref{energy}) holds, that is, for
any $m\geq0$,
\begin{align} \label{energy_discrete}
\frac{E^{m+1}-E^{m}}{\triangle
t}&=-\norm{ p^{m+1/2}-n^{m+1/2}}^{2}_{h}-\norm{\sqrt{p^{m}_{L}+n^{m}_{L}}\partial^{h}_{x}\phi^{m+1/2}}^{2}_{h}-\norm{\sqrt{p^{m}_{R}+n^{m}_{R}}\partial^{h}_{y}\phi^{m+1/2}}^{2}_{h},
\end{align}
where
$E^{m}=\frac{1}{2}\parallel\nabla\phi^{m}\parallel^{2}_{h}$,
$p^{m}_{L}$ and $p^{m}_{R}$ denotes the average value of $p^{m}$ in
x and y direction, i.e.
$$(p^{m}_{L})_{j,k}=\frac{p^{m}_{j-1,k}+p^{m}_{j,k}}{2},\quad (p^{m}_{R})_{j,k}=\frac{p^{m}_{j,k-1}+p^{m}_{j,k}}{2},$$
and similar for $n^{m}_{L}$ and $n^{m}_{R}$.
\end{theorem}
{\bf Proof.} Equation (\ref{phi_discrete}) of time level $m$ minus equation
(\ref{phi_discrete}) of time level $m-1$ gives,
\begin{equation} \label{energy_discrete2}
-\triangle_{h}(\phi^{m+1}_{j,k}-\phi^{m}_{j,k})=(p^{m+1}_{j,k}-p^{m}_{j,k})-(n^{m+1}_{j,k}-n^{m}_{j,k}).
\end{equation}
Multiplying above equation with $\frac{1}{\triangle
t}\phi^{m+1/2}_{j,k}\triangle x\triangle y$ to both
sides and summing for $j=1\cdots N_{x}, k=1\cdots N_{y}$, and
applying the boundary conditions (\ref{0_Neumann}), we get
\begin{align}
&<\frac{p^{m+1}-p^{m}}{\triangle t},{\phi^{m+1/2}}>_{h}-<\frac{n^{m+1}-n^{m}}{\triangle
t},{\phi^{m+1/2}}>_{h}\nonumber\\
&=-\frac{1}{2\triangle t}\sum^{N_{x}}_{j
=1}\sum^{N_{y}}_{k=1}\triangle_{h}(\phi^{m+1}_{j,k}-\phi^{m}_{j,k})(\phi^{m+1}_{j,k}+\phi^{m}_{j,k})\triangle
x\triangle y\nonumber\\
&=\frac{1}{2\triangle
t}\sum^{N_{x}}_{j=1}\sum^{N_{y}}_{k=1}\nabla_{h}(\phi^{m+1}_{j,k}-\phi^{m}_{j,k})\cdot\nabla_{h}(\phi^{m+1}_{j,k}+\phi^{m}_{j,k})\triangle
x\triangle y\nonumber\\
&=\frac{1}{2\triangle
t}\sum^{N_{x}}_{j=1}\sum^{N_{y}}_{k=1}|\nabla_{h}\phi^{m+1}_{j,k}|^{2}-|\nabla_{h}\phi^{m}_{j,k}|^{2}\triangle
x\triangle y\nonumber\\
&=\frac{E^{m+1}-E^{m}}{\triangle t}.
\end{align}

Substituting (\ref{p_discrete}) and (\ref{n_discrete}) into above
equation, we get
\begin{align}\label{E_discrete}
\frac{E^{m+1}-E^{m}}{\triangle
t}&=\left<\triangle_{h}p^{m+1/2}+\left(\delta_x\left(p^m\delta_x\phi^{m+1/2}\right)+\delta_y\left(p^m\delta_y\phi^{m+1/2}\right)\right),\phi^{m+1/2}\right>_{h}\nonumber\\
&\;\;\;\;-\left<\triangle_{h}n^{m+1/2}-\left(\delta_x\left(n^m\delta_x\phi^{m+1/2}\right)+\delta_y\left(n^m\delta_y\phi^{m+1/2}\right)\right),\phi^{m+1/2}\right>_{h}\nonumber\\
&=\left(\left<p^{m+1/2},\triangle_{h}\phi^{m+1/2}\right>_{h}-\norm{\sqrt{p^{m}_{L}}\partial^{h}_{x}\phi^{m+1/2}}^{2}_{h}-\norm{\sqrt{p^{m}_{R}}\partial^{h}_{y}\phi^{m+1/2}}^{2}_{h}\right)-\nonumber\\
&\;\;\;\;\left(\left<n^{m+1/2},\triangle_{h}\phi^{m+1/2}\right>_{h}+\norm{\sqrt{n^{m}_{L}}\partial^{h}_{x}\phi^{m+1/2}}^{2}_{h}+\norm{\sqrt{n^{m}_{R}}\partial^{h}_{y}\phi^{m+1/2}}^{2}_{h}\right)\nonumber\\
&=\left<p^{m+1/2}-n^{m+1/2},\triangle_{h}\phi^{m+1/2}\right>_{h}-\norm{\sqrt{p^{m}_{L}+n^{m}_{L}}\partial^{h}_{x}\phi^{m+1/2}}^{2}_{h}-\norm{\sqrt{p^{m}_{R}+n^{m}_{R}}\partial^{h}_{y}\phi^{m+1/2}}^{2}_{h}.
\end{align}
Here we have used
\begin{align}\label{E_discrete1}
\left<\triangle_{h}p^{m+1/2},\phi^{m+1/2}\right>_{h}&=-\left<\nabla_{h}p^{m+1/2},\nabla_{h}\phi^{m+1/2}\right>_{h}=\left<p^{m+1/2},\triangle_{h}\phi^{m+1/2}\right>_{h},\\
\left<\triangle_{h}n^{m+1/2},\phi^{m+1/2}\right>_{h}&=-\left<\nabla_{h}n^{m+1/2},\nabla_{h}\phi^{m+1/2}\right>_{h}=\left<n^{m+1/2},\triangle_{h}\phi^{m+1/2}\right>_{h},
\end{align}
and
\begin{align}\label{E_discrete21}
\left<\delta_x\left(p^m\delta_x\phi^{m+1/2}\right),\phi^{m+1/2}\right>_{h}=-\norm{\sqrt{p^{m}_{L}}\partial^{h}_{x}\phi^{m+1/2}}^{2}_{h},
\end{align}

\begin{align}\label{E_discrete22}
\left<\delta_y\left(p^m\delta_y\phi^{m+1/2}\right),\phi^{m+1/2}\right>_{h}=-\norm{\sqrt{p^{m}_{R}}\partial^{h}_{y}\phi^{m+1/2}}^{2}_{h},
\end{align}

\begin{align}\label{E_discrete23}
\left<\delta_x\left(n^m\delta_x\phi^{m+1/2}\right),\phi^{m+1/2}\right>_{h}=-\norm{\sqrt{n^{m}_{L}}\partial^{h}_{x}\phi^{m+1/2}}^{2}_{h},
\end{align}

\begin{align}\label{E_discrete33}
\left<\delta_y\left(n^m\delta_y\phi^{m+1/2}\right),\phi^{m+1/2}\right>_{h}=-\norm{\sqrt{n^{m}_{R}}\partial^{h}_{y}\phi^{m+1/2}}^{2}_{h}.
\end{align}
Equations (\ref{E_discrete1})-(\ref{E_discrete33}) can be easily
checked when applying the boundary conditions (\ref{0_Neumann}).

Equation (\ref{phi_discrete}) of time level $m$ plus equation
(\ref{phi_discrete}) of time level  $m-1$ gives,
\begin{equation} \label{discrete}
-\triangle_{h}(\phi^{m+1}_{j,k}+\phi^{m}_{j,k})=(p^{m+1}_{j,k}+p^{m}_{j,k})-(n^{m+1}_{j,k}+n^{m}_{j,k}).
\end{equation}
Substituting (\ref{discrete}) into (\ref{E_discrete}), we get equation (\ref{energy_discrete}).
This completes the proof of Theorem \ref{thm3}.

\section{Extending the method to the Poisson-Nernst-Planck system with multi-ions}\label{sec:4}
In this section, we shall extend the above method to the case of
multi-ions. The model equations are as follows,
\begin{equation}\label{c_i}
(c_{i})_{t}=\nabla\cdot(\nabla c_{i}+z_{i}c_{i}\nabla\phi), \
\textrm{in}\ \Omega,
\end{equation}
\begin{equation}\label{potentialm}
-\Delta\phi=\sum_{i}z_{i}c_{i}, \ \textrm{in}\ \Omega,
\end{equation}
where $c_{i}$ is a ion with valence $z_{i}$. Using a similar method
as in the introduction part, it is easy to check the above system
satisfy the following energy and mass identities under zero Neumann
boundary conditions,
\begin{equation}\label{energym}
\frac{dE}{dt}=-\int_{\Omega}\Big(\big\vert\sum_{i}z_{i}c_{i}\big\vert^{2}+(\sum_{i}z^{2}_{i}c_{i})|\nabla\phi|^{2}\Big)
d\vec{x},
\end{equation}
\begin{equation} \label{massm}
\frac{d}{dt}\int_{\Omega}c_{i}(t,\vec{x})d\vec{x}=0.
\end{equation}

The scheme for the above Poisson-Nernst-Planck system is as follows
\begin{align}
\frac{(c_{i})^{m+1}_{j,k}-(c_{i})^{m}_{j,k}}{\Delta
t}&=\Delta_{h}(c_{i})^{m+1/2}_{j,k}+z_{i}\left(\delta_x\left((c_i)^{m}_{j,k}\delta_x\phi^{m+1/2}_{j,k}\right)+\delta_y\left((c_i)^{m}_{j,k}\delta_y\phi^{m+1/2}_{j,k}\right)\right),
\label{p_discretm}
\\
-\Delta_{h}\phi^{m+1}_{j,k}&=\sum_{i}z_{i}(c_{i})^{m+1}_{j,k},\label{phi_discretem}
\end{align}
where $$\phi^{m+1/2}_{j,k}=\frac{\phi^{m}_{j,k}+\phi^{m+1}_{j,k}}{2},\quad (c_{i})^{m+1/2}_{j,k}=\frac{(c_{i})^{m+1}_{j,k}+(c_{i})^{m}_{j,k}}{2}.$$

The matrix and vector form of scheme
(\ref{p_discretm})-(\ref{phi_discretem}), is as follows,
\begin{align} \label{p_matrixm}
(\frac{I}{\Delta t}-\frac{F}{2})C^{m+1}_{i}=(\frac{I}{\Delta
t}+\frac{F}{2})C^{m}_{i}+z_{i}A(C^{m}_{i})\frac{\Phi^{m+1}+\Phi^{m}}{2},
\end{align}
\begin{equation} \label{phi_matrixm}
F\Phi^{m+1}=\sum_{i}z_{i}C^{m+1}_{i}.
\end{equation}
(\ref{p_matrixm}) and (\ref{phi_matrixm}) are linear system of
$C^{m+1}_{i}$ and $\Phi^{m+1}$, which can be solved efficiently,
since all these matrices are symmetric and banded.

Using similar techniques as in Theorem \ref{thm1} and \ref{thm3}, we could also
prove that the above numerical scheme satisfies  the mass conservation and energy decay properties, which is stated in the following theorem.\\

\begin{theorem}
For the solutions of (\ref{p_discretm})-(\ref{phi_discretem}), the
discrete form of mass conservation (\ref{massm}) and energy identity
(\ref{energym}) holds, that is, for any $m\geq0$,
\begin{equation} \label{mass_discretem}
\sum^{N_{x}}_{j=1}\sum^{N_{y}}_{k=1}(c_{i})^{m+1}_{j,k}\Delta
x\Delta
y=\sum^{N_{x}}_{j=1}\sum^{N_{y}}_{k=1}(c_{i})^{m}_{j,k}\Delta
x\Delta y.
\end{equation}
Moreover, the following energy identity is preserved:
\begin{align} \label{energy_discretem}
\frac{E^{m+1}-E^{m}}{\Delta
t}=&-\norm{\frac{\sum_{i}z_{i}c^{m}_{i}+\sum_{i}z_{i}c^{m+1}_{i}}{2}}^{2}_{h}-\norm{\sqrt{\sum_{i}z^{2}_{i}c^{m}_{iL}}{\partial^{h}_{x}\phi^{m+1/2}}}^{2}_{h}\nonumber\\
&-\norm{\sqrt{\sum_{i}z^{2}_{i}c^{m}_{iR}}{\partial^{h}_{y}\phi^{m+1/2}}}^{2}_{h},
\end{align}
where $E^{m}=\frac{1}{2}\norm{\nabla\phi^{m}}^{2}_{h}$, $c^{m}_{iL}$ and $c^{m}_{iR}$ denotes the average value of
$c^{m}_{i}$ in $x$ and $y$ directions, i.e.
$$(c^{m}_{iL})_{j,k}=\frac{(c_{i})^{m}_{j-1,k}+(c_{i})^{m}_{j,k}}{2},\;\;\; (c^{m}_{iR})_{j,k}=\frac{(c_{i})^{m}_{j,k-1}+(c_{i})^{m}_{j,k}}{2}.$$
\end{theorem}

\section{Numerical results}\label{sec:5}

For the sake of simplicity, we only give examples for the
Poisson-Nernst-Planck system with two ions, i.e.
(\ref{potential}).

From the numerical scheme (\ref{p_discrete})-(\ref{phi_discrete}), it is easy to check that the truncation error of (\ref{p_discrete}) and (\ref{n_discrete}) are $O(\Delta t+ (\Delta x)^2+ (\Delta y)^2)$, while (\ref{phi_discrete}) has the truncation error $O( (\Delta x)^2+ (\Delta y)^2)$. Thus the numerical scheme is expected to convergent with first order in time and second order in space.

{\bf Example 1}. Since it is not possible to find  the exact solutions for the equations
(\ref{potential}), we are now use the following argumented equations with exact solutions as a test problem:
\begin{equation}\label{p1}
p_{t}=\nabla\cdot(\nabla p+p\nabla\phi)+f_1, \ \textrm{in}\  \Omega_T=[0,T]\times\Omega,
\end{equation}
\begin{equation}\label{n1}
n_{t}=\nabla\cdot(\nabla n-n\nabla\phi)+f_2, \ \textrm{in}\ \Omega_T=[0,T]\times\Omega,
\end{equation}
\begin{equation}\label{potential1}
-\Delta\phi=p-n+c, \ \textrm{in}\ \Omega_T=[0,T]\times\Omega,
\end{equation}
where
\begin{equation*}
 p=(3x^2-2x^3+3y^2-2y^3)e^{-t}, n=(x^2(1-x)^2+y^2(1-y)^2)e^{-t}, \phi=x^2(1-x)^2y^2(1-y)^2e^{-t}
\end{equation*}
 are the exact solutions of (\ref{p1})-(\ref{potential1}), which satisfy the zero Neumann boundary conditions (\ref{boundary}). And $f_1$, $f_2$, $c$ are known functions which are given according to these exact solutions.

We do the discretization of equations (\ref{p1})-(\ref{potential1}) as follows:

\begin{align} \label{p_discrete1}
\frac{p^{m+1}_{j,k}-p^{m}_{j,k}}{\Delta
t}=\Delta_{h}p^{m+1/2}_{j,k}+\left(\delta_x(p^m_{j,k}\delta_x\phi^{m+1/2}_{j,k})+\delta_y(p^m_{j,k}\delta_y\phi^{m+1/2}_{j,k})\right)+(f_1)^{m}_{j,k},
\end{align}
\begin{align} \label{n_discrete1}
\frac{n^{m+1}_{j,k}-n^{m}_{j,k}}{\Delta
t}=\Delta_{h}n^{m+1/2}_{j,k}-\left(\delta_x(n^m_{j,k}\delta_x\phi^{m+1/2}_{j,k})+\delta_y(n^m_{j,k}\delta_y\phi^{m+1/2}_{j,k})\right)+(f_2)^{m}_{j,k},
\end{align}
\begin{equation} \label{phi_discrete1}
-\Delta_{h}\phi^{m+1}_{j,k}=p^{m+1}_{j,k}-n^{m+1}_{j,k}+c^{m+1}_{j,k}.
\end{equation}
In the numerical computation, since the electric potential $\phi$ is not unique up to a constant, we set electric potential at the first point to be the exact value at each time step in order to get unique solutions.
It is easy to check that  the truncation error of the above discretization scheme for the system (\ref{potential1}) has  the same order as  the truncation error  of  the discretization scheme (\ref{p_discrete})-(\ref{phi_discrete}) for the problem (\ref{potential}). For this example, when numerically implementing of (\ref{p_discrete1})-(\ref{phi_discrete1}), we set $\phi^{m+1}$ at one boundary point to be the exact value at each time step so that $\phi^{m+1}$ is uniquely determined.

Now we  carry
out the numerical convergence study for both space and time using (\ref{p_discrete1})-(\ref{phi_discrete1}). For spatial convergence,
 we set  $\Delta t=0.000002$, and use 4 different spatial meshes  $\Delta x=\Delta y=\frac{1}{20\times 2^{n}}\triangleq h$, $n=0,\cdots,3$, the final time is set to be $T=1.0$.
 When $\Delta t$ is sufficiently small, we compute the spatial
convergence order according to
\begin{equation} \label{order1}
\textrm{order1}=\log_2\frac{||u_{h}(\cdot,\cdot,T)-u_{exact}(\cdot,\cdot,T)||}{||u_{\frac{h}{2}}(\cdot,\cdot,T)-u_{exact}(\cdot,\cdot,T)||},
\end{equation}
where $u_{h}(\cdot,\cdot,T)$ is the numerical solution at time $t=T$ using mesh $h$, $u_{exact}(\cdot,\cdot,T)$
is the exact solution at time $t=T$, and $||\cdot||$ is the spatial discrete norm.
Table \ref{table1} shows the mesh refinement analysis for $p,n,\phi$
using two different norms. One can
see, the errors are decreasing when spatial mesh is refined, and it
is second order convergent for both norms, which is expected from the truncation error analysis.

\begin{table}[htbp]
\tabcolsep=4pt
\begin{center}
\caption{Spatial mesh refinement analysis for the Poisson-Nernst-Planck
system with zero Neumann boundary conditions ($e_p=p_h-p_{exact}, e_n=n_h-n_{exact}, e_{\phi}=\phi_h-\phi_{exact}, \Delta t=0.000002$).}\label{table1}
\begin{tabular}{lcccccccccccc}
\hline\noalign{\smallskip}
$h$&$\norm{e_p}_{2}$&order1&$\norm{
e_p}_{\infty}$&order1&$\norm{
e_n}_{2}$&order1&$\norm{
e_n}_{\infty}$&order1&$\norm{
e_{\phi}}_{2}$&order1&$\norm{
e_{\phi}}_{\infty}$&order1\\
\noalign{\smallskip}\hline\noalign{\smallskip}
 ${1}/{20} $&3.48e-4 &      -&7.65e-4&   -&3.16e-3&   -& 3.24e-3&    -&5.19e-3&   -&5.74e-3&   -\\
 ${1}/{40} $&8.71e-5 & 2.00  &1.94e-4&1.98&7.89e-4&2.00& 8.09e-4& 2.00&1.63e-3&1.67&1.77e-3&1.70\\
 ${1}/{80} $&2.17e-5 & 2.00  &4.91e-5&1.98&1.97e-4&2.00& 2.02e-4& 2.00&4.87e-4&1.74&5.23e-4&1.76\\
 ${1}/{160}$&5.42e-6 & 2.00  &1.25e-5&1.97&4.94e-5&2.00& 5.06e-5& 2.00&1.40e-4&1.80&1.49e-4&1.81\\
\noalign{\smallskip}\hline
\end{tabular}
\end{center}
\end{table}

\begin{table}[htbp]
\tabcolsep=4pt
\centering
\caption{Temporal mesh refinement analysis with $h=1/640$ for the Poisson-Nernst-Planck
system with zero Neumann boundary conditions.}\label{table2}
\begin{tabular}{lcccccccccccc}
\hline\noalign{\smallskip}
 $\Delta t$&$\norm{e_p}_{2}$&order2&$\norm{
e_p}_{\infty}$&order2&$\norm{
e_n}_{2}$&order2&$\norm{
e_n}_{\infty}$&order2&$\norm{
e_{\phi}}_{2}$&order2&$\norm{
e_{\phi}}_{\infty}$&order2\\
\noalign{\smallskip}\hline\noalign{\smallskip}
 ${1}/{40} $ &8.26e-3 &    -&1.31e-2&   -&5.46e-4&   -& 8.61e-4 &    -&6.51e-3&    -&7.40e-3&-   \\
 ${1}/{80} $ &4.14e-3 & 1.00&6.58e-3&1.00&2.75e-4&0.99& 4.33e-4 & 0.99&4.05e-3&0.69 &4.52e-3&0.71\\
 ${1}/{160}$ &2.07e-3 & 1.00&3.29e-3&1.00&1.39e-4&0.98& 2.18e-4 & 0.99&2.42e-3&0.74 &2.67e-3&0.76\\
 ${1}/{320}$ &1.04e-3 & 1.00&1.65e-3&1.00&7.11e-5&0.97& 1.11e-4 & 0.98&1.41e-3&0.78 &1.53e-3&0.80\\
\noalign{\smallskip}\hline
\end{tabular}
\end{table}

 For time convergence,
we set $\Delta x=\Delta y=h=1/640$, and use 4 different time steps
$\Delta t=\frac{1}{40\times 2^{n}}$, $n=0,\cdots,3$, the final time
is set to be $T=1.0$. When $h$ is sufficiently small, we compute the temporal
convergence order according to
\begin{equation} \label{order2}
\textrm{order2}=\log_2\frac{||u_{\Delta t}(\cdot,\cdot,T)-u_{exact}(\cdot,\cdot,T)||}{||u_{{\Delta t}/{2}}(\cdot,\cdot,T)-u_{exact}(\cdot,\cdot,T)||},
\end{equation}
where $u_{\Delta t}(\cdot,\cdot,T)$ is the numerical solution at time $t=T$ using the time step $\Delta t$.
Table \ref{table2} shows the
time step refinement analysis for $p,n,\phi$
using two different norms. One can
see, the errors are decreasing when time step is refined, and it
is first order convergent in time,  which also is expected from the truncation error analysis.

\begin{figure}[!hbtp]
\tabcolsep=4pt
 \centering
 \subfigure[electric potential energy $E$ w.r.t. time]{
\label{fig:subfig1:1} 
\includegraphics[scale=0.4]{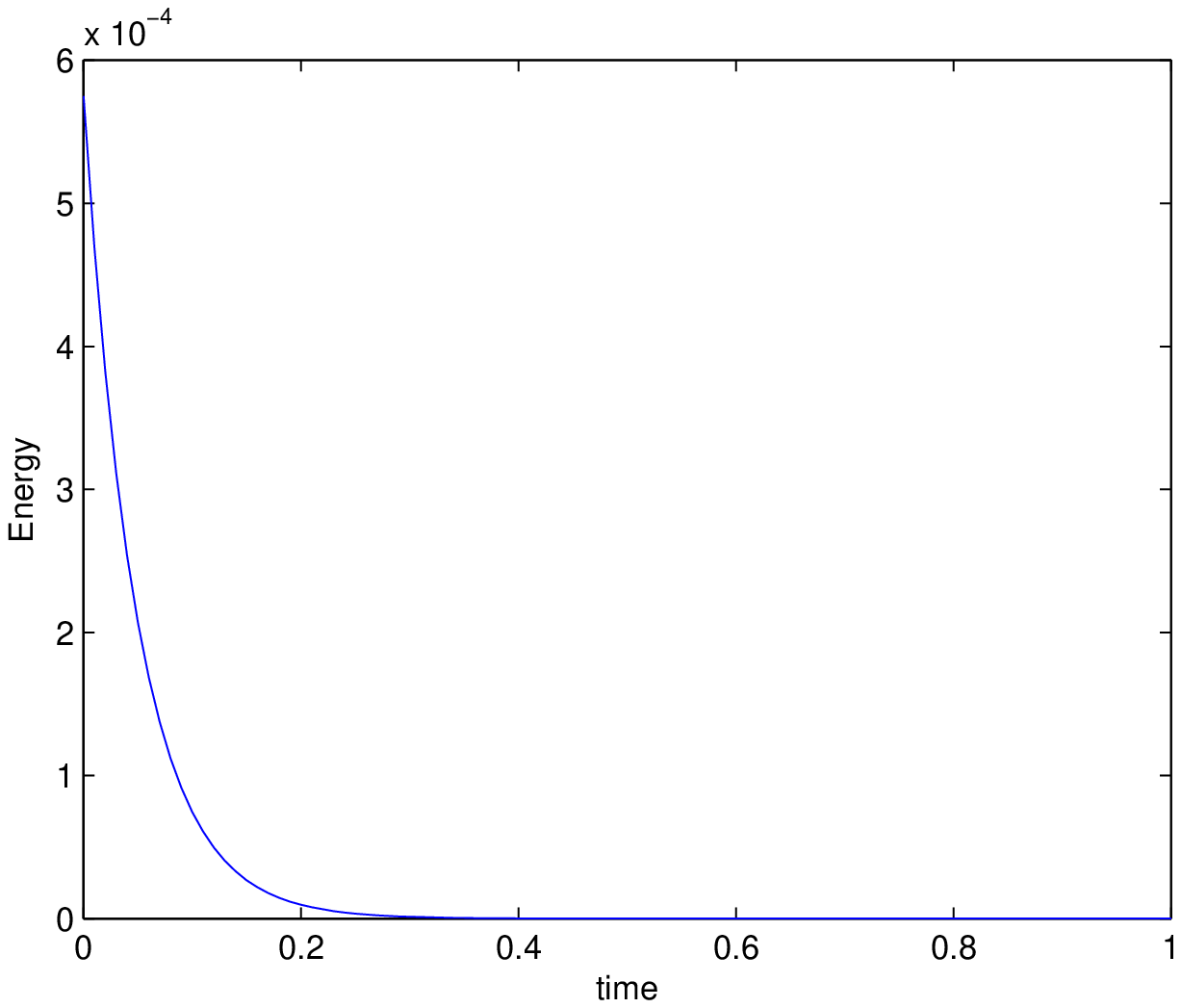}}
\hspace{0.1in}
\subfigure[mass of $p$ w.r.t. time]{
\label{fig:subfig1:2} 
\includegraphics[scale=0.4]{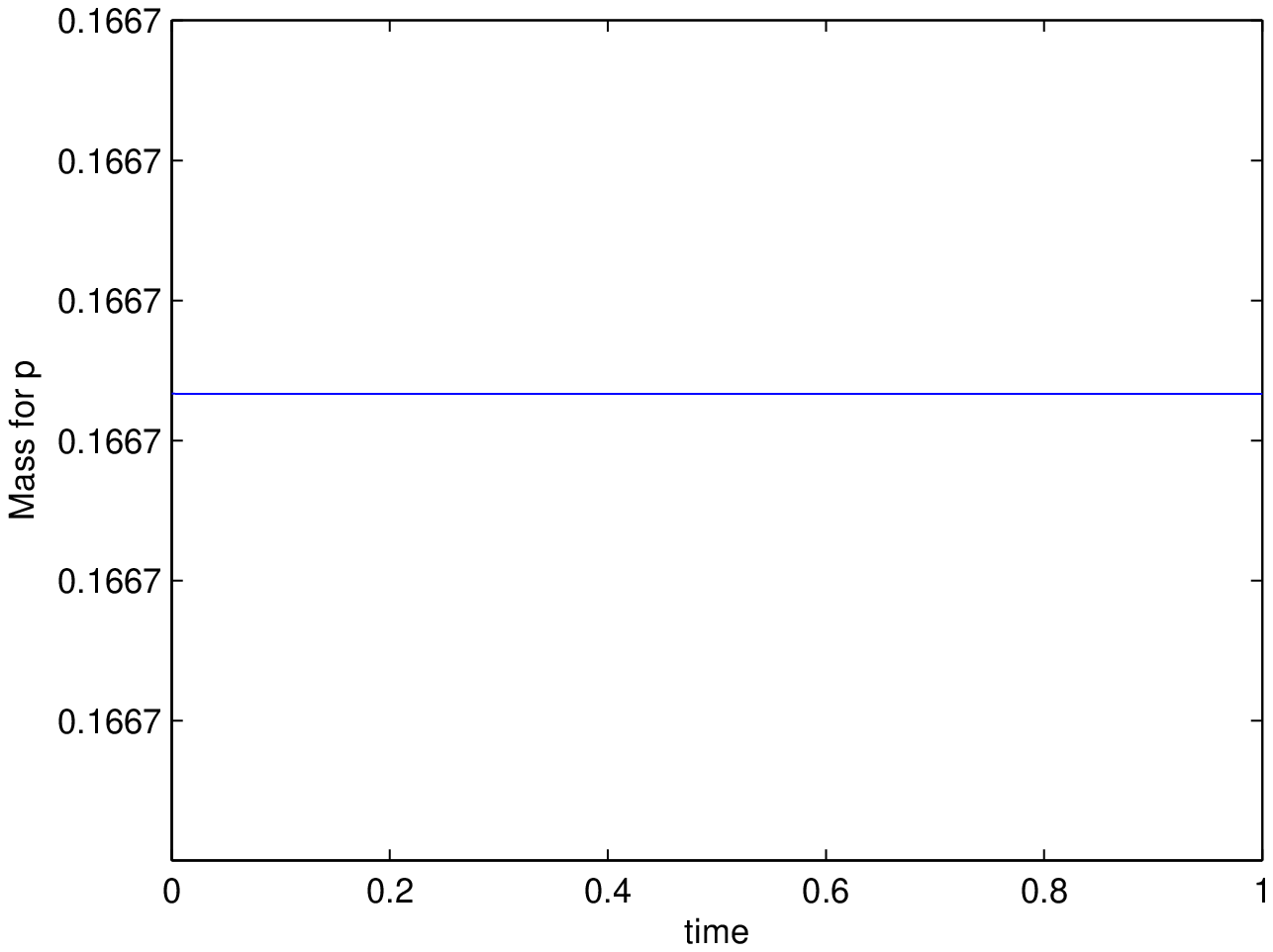}}\\
\subfigure[mass of $n$ w.r.t. time]{
\label{fig:subfig1:3} 
\includegraphics[scale=0.4]{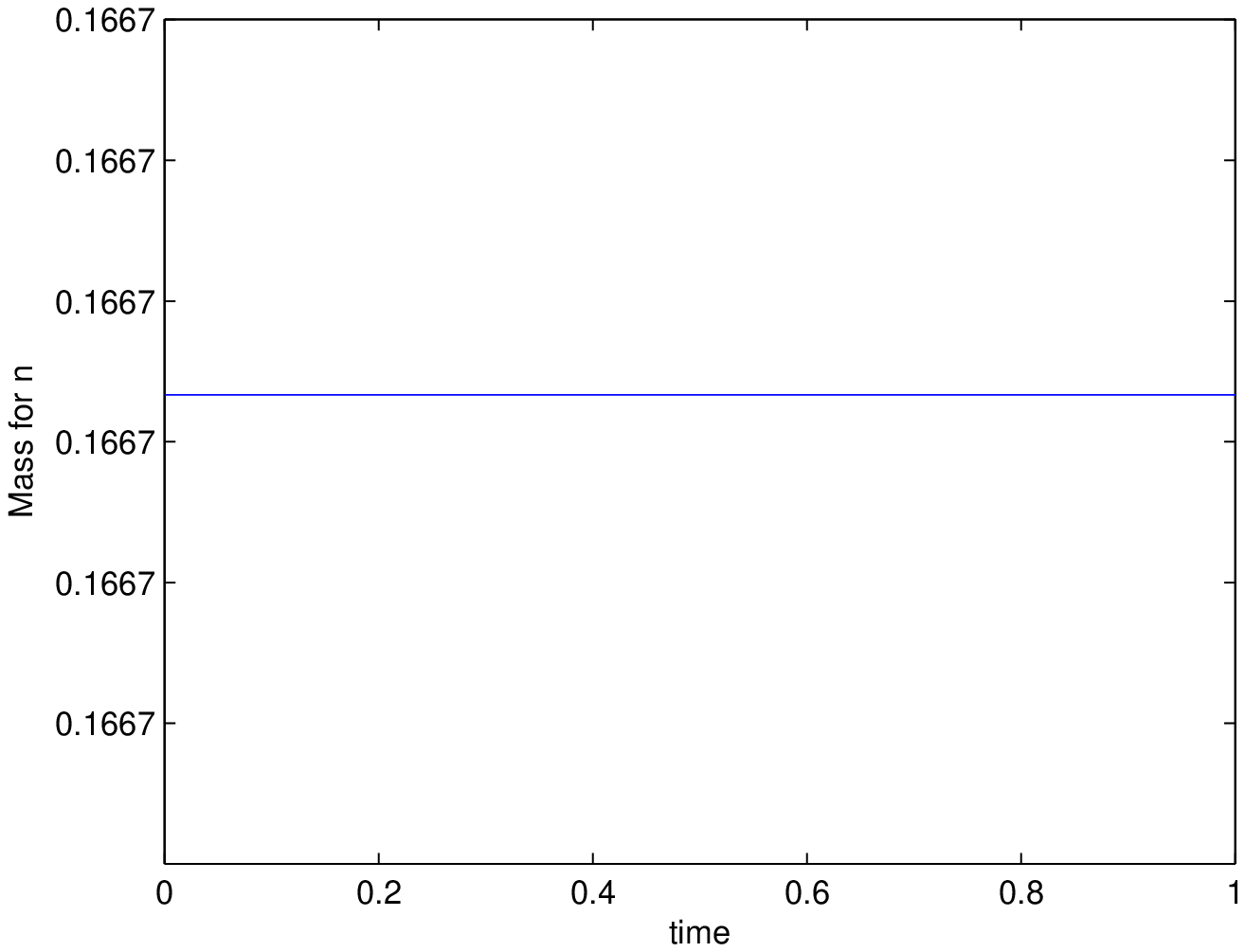}}
\hspace{0.1in}
\subfigure[ relative mass error of $p$ w.r.t. time]{
\label{fig:subfig1:4} 
\includegraphics[scale=0.4]{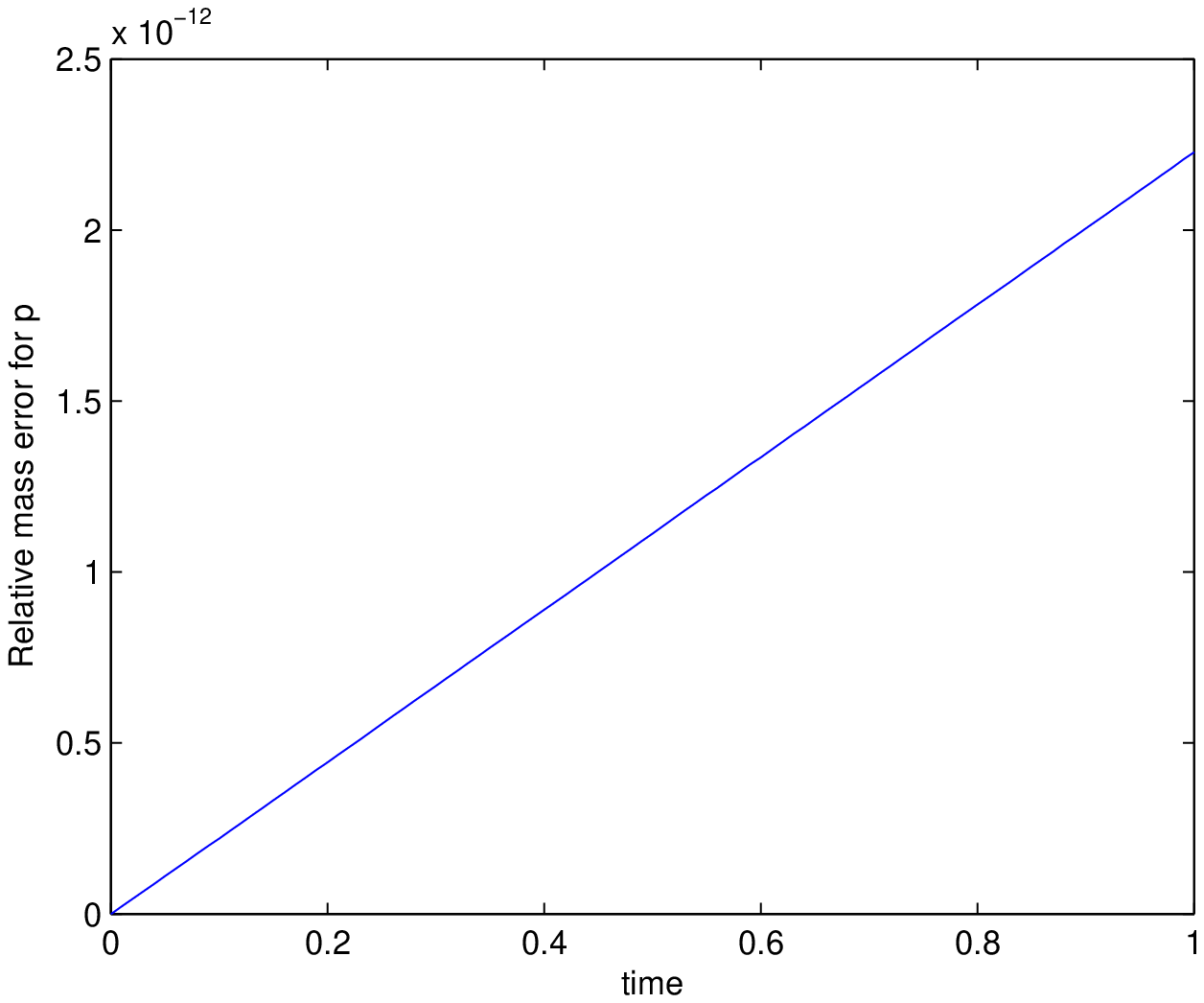}}\\
\subfigure[relative mass error of $n$ w.r.t. time]{
\label{fig:subfig1:5} 
\includegraphics[scale=0.4]{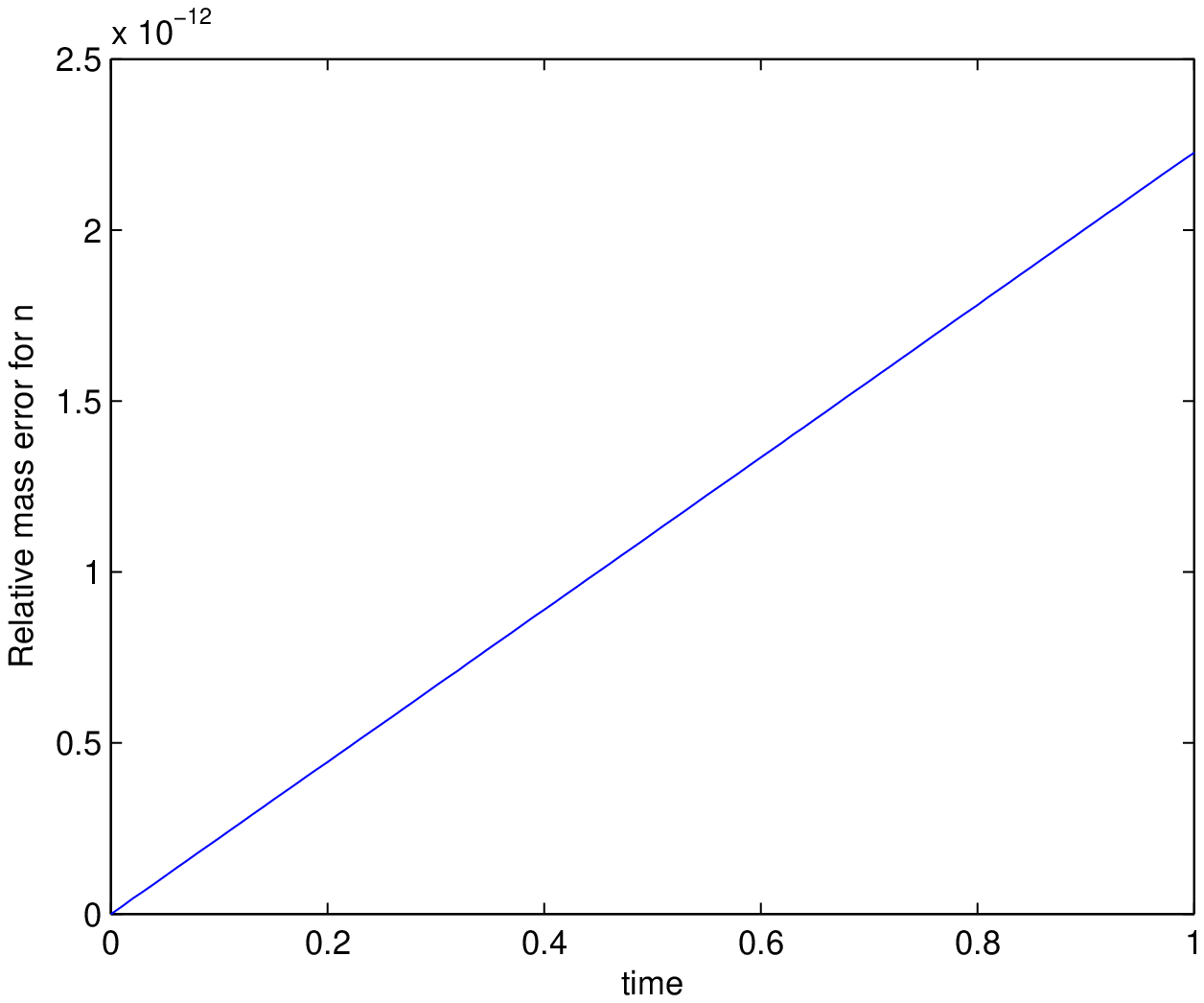}}
\hspace{0.05in}
\subfigure[minimum $p$ in $\Omega$ w.r.t. time]{
\label{fig:subfig1:6} 
\includegraphics[scale=0.4]{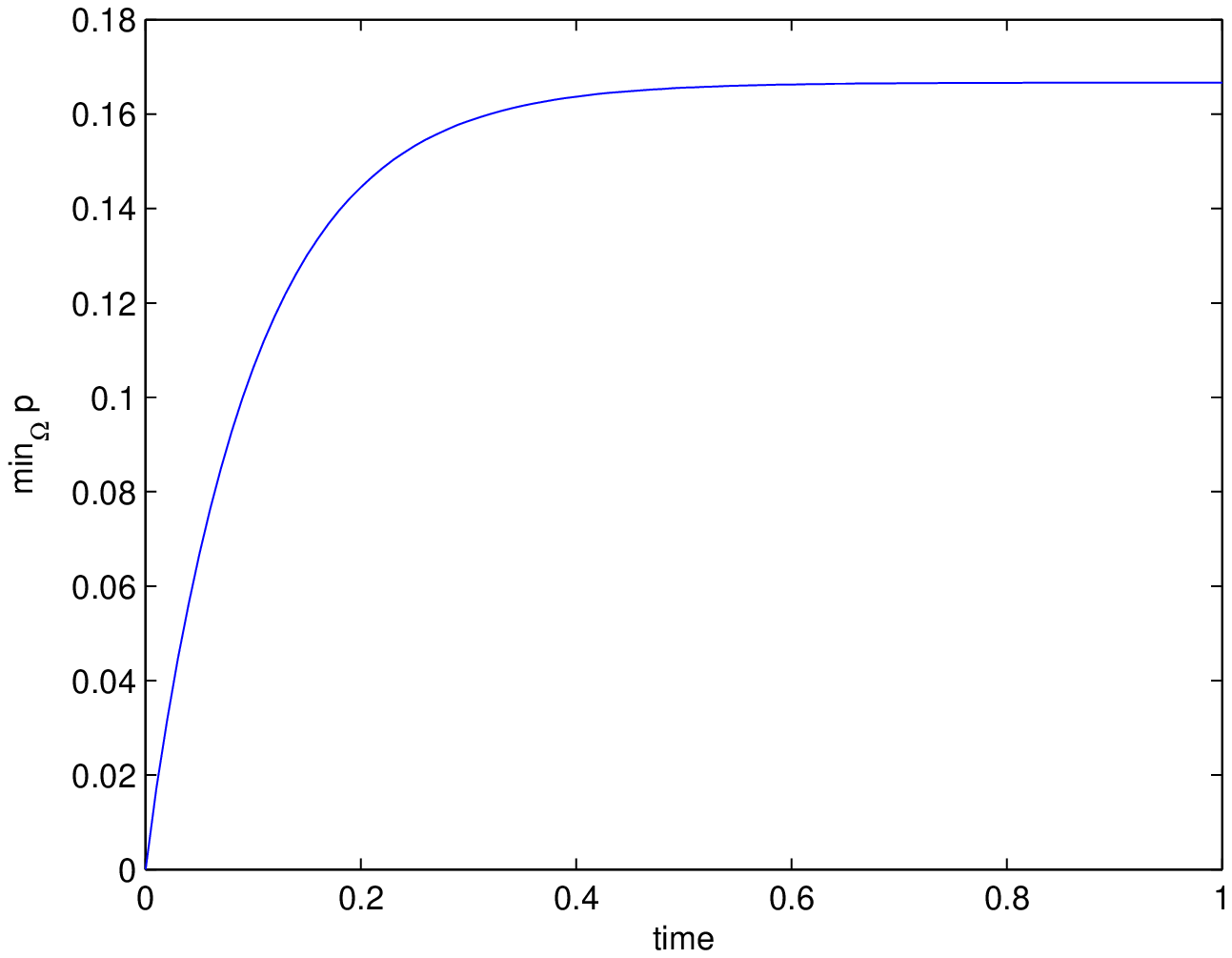}}
\subfigure[minimum $n$ in $\Omega$ w.r.t. time]{
\label{fig:subfig1:7} 
\includegraphics[scale=0.4]{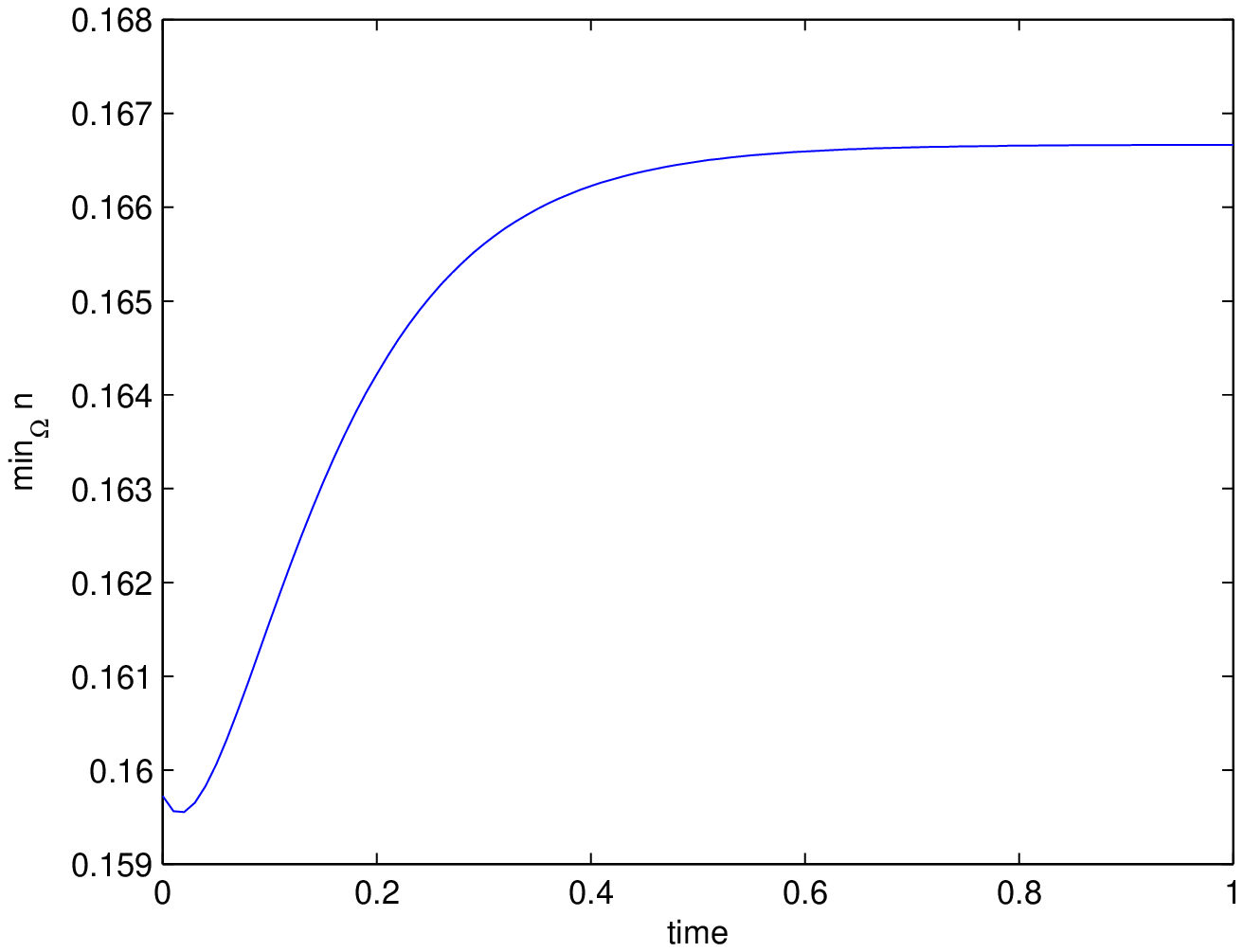}}
\caption{Numerical results for the Poisson-Nernst-Planck system with zero Neumann boundary conditions.}
\label{fig1}
\end{figure}

{\bf Example 2}. We consider the equations
(\ref{potential}) in the domain $\Omega=[0,1]\times[0,1]$
with zero Neumann boundary conditions (\ref{boundary}), and initial
conditions
\begin{align*}
  p(0,x,y)&=\frac{x^{2}}{2}-\frac{x^{3}}{3}+\frac{y^{2}}{2}-\frac{y^{3}}{3},\\
n(0,x,y)&=(\frac{x^{2}}{2}-\frac{x^{3}}{3})(\frac{y^{2}}{2}-\frac{y^{3}}{3})+\frac{23}{144},
\end{align*}
where the initial conditions are set to satisfy the conditions
(\ref{initial})-(\ref{mass}).

We carry out numerical computation with $h=\frac{1}{108}, \Delta t=0.01, T=1$ using the scheme (\ref{p_discrete})- (\ref{phi_discrete}). In the numerical computation, we set electric potential at the first point to be zero at each time step in order to get unique solutions. Figure \ref{fig1} gives the evolution of electric potential energy, mass of $p$, mass of $n$, the relative mass error of $p$, the relative mass error of $n$, $\min_{\Omega} (p)$,  $\min_{\Omega} (n)$, respectively.  One can see that the
electric potential energy decays, and the mass of $p,n$ is exactly
conserved. Moreover, $p$ and $n$ always keep positive. All these results are consistent with the analysis in above
section.

\section{Conclusions and discussions}\label{sec:6}
Prohl et al.~\cite{Prohl} first proposed a fully implicit finite
element method for the Poisson-Nernst-Planck system. Numerically, in order to get the rigorous mass conservation and electric potential  energy decay properties, a fixed iteration method is needed for the fully implicit finite element scheme.
In this paper, we  develop a simple semi-implicit finite difference method
for the Poisson-Nernst-Planck system, which can also preserve mass and
electric potential energy identities numerically. The current method only needs to solve a
linear system at each time step, which can be done very efficiently since the all
coefficient matrices are symmetric banded. Furthermore, mesh
refinement analysis shows that the method is second order convergent in space and
first order convergent in time. And the method can be easily extended to the case of
multi-ions.

Since the Poisson-Nernst-Planck system is a nonlinear
system, theoretical convergence analysis for the proposed numerical
method will be a challenge task, we leave it as the future work.
Moreover, constructing simple and  efficient numerical
method  which can preserve the entropy law of the Poisson-Nernst-Planck system is another future goal.

\section*{Acknowledgements}
Dongdong He was supported by the Program for Young Excellent Talents at Tongji University (No. 2013KJ012), the Natural Science Foundation of China (No. 11402174) and the Scientific Research Foundation for the Returned Overseas Chinese Scholars, State Education Ministry. Kejia Pan was supported by the Natural Science Foundation of China (Nos. 41474103, 41204082),
the National High Technology Research and Development Program of China (No. 2014AA06A602),
the Natural Science Foundation of Hunan Province of China (No. 2015JJ3148) and Mathematics and Interdisciplinary Sciences Project of Central South University. The authors would like to thank Professor Huaxiong Huang for useful discussions.



\end{document}